\documentclass[11pt]{article}

\usepackage[a4paper,margin=1.15in]{geometry}
\usepackage[T1]{fontenc}
\usepackage{lmodern}
\usepackage{microtype}
\usepackage{xcolor}
\usepackage{amsmath,amssymb,amsthm,mathtools}
\usepackage[numbers,sort&compress]{natbib}
\usepackage[colorlinks=true,linkcolor=blue!50!black,
           citecolor=blue!50!black,urlcolor=blue!50!black,]{hyperref}
\usepackage[nameinlink,capitalize,noabbrev]{cleveref}
\usepackage{authblk}
\usepackage{mathrsfs}

\newtheorem{theorem}{Theorem}[section]
\newtheorem{lemma}[theorem]{Lemma}

\newtheorem{conjecture}[theorem]{Conjecture}
\newtheorem{observation}{Observation}[section]
\newtheorem{claim}{Claim}[section]
\theoremstyle{remark}

\newcommand{\ex}{\operatorname{ex}}

\allowdisplaybreaks

\title{\bf A subquadratic bound for generalized Tur\'an numbers of odd cycles}
\author{
Zhen Liu\footnote{Email: 1552580575@qq.com},
~Chuanshu Wu\footnote{Email: cswu97@126.com (Corresponding author)}\\
{\small Center for Discrete Mathematics, Fuzhou University, Fujian 350003, China}}

\date{\today}

\begin{document}

\maketitle

\begin{abstract}
For a graph $H$ and a family of graphs $\mathcal F$, let $\ex(n,H,\mathcal F)$ denote the maximum number of copies of $H$ in an $\mathcal F$-free graph on $n$ vertices. For every integer $i\ge 3$, let $C_i$ denote the cycle of length $i$. For $r\ge 3$, set $\mathscr {C}_r=\{C_3,C_4,\ldots,C_r\},$ and set $\mathscr {C}_2=\varnothing$. In this paper, we prove that, for all integers $l>k\ge 2$,
$$
\ex(n,C_{2k+1},\mathscr {C}_{2k}\cup\{C_{2l+1}\})
=O_{k,l} \left(n^{\,2-\frac{1}{k(k+1)(l-k)}}\right).
$$
Together with the known upper bounds for the number of triangles in $C_{2l+1}$-free graphs, this confirms a conjecture of Gerbner, Gy\H{o}ri, Methuku, and Vizer.
\end{abstract}



\section{Introduction}

Throughout this paper, all graphs are finite and simple. For a family of graphs $\mathcal F$, a graph $G$ is called $\mathcal F$-free if it contains no member of $\mathcal F$ as a subgraph. The classical Tur\'an number is
\[
\ex(n,\mathcal F)
 =\max\bigl\{|E(G)|:|V(G)|=n\text{ and $G$ is $\mathcal F$-free}\bigr\}.
\]
For graphs $H$ and $G$, let $N(H,G)$ denote the number of copies of $H$ in $G$, counted as unlabeled, not necessarily induced subgraphs. The generalized Tur\'an number of $H$ with respect to $\mathcal F$ is
\[
\ex(n,H,\mathcal F)
 =\max\bigl\{N(H,G):|V(G)|=n\text{ and $G$ is $\mathcal F$-free}\bigr\}.
\]
Let $K_r$ denote the complete graph on $r$ vertices. Thus $\ex(n,K_2,\mathcal F)=\ex(n,\mathcal F)$. The clique case $\ex(n,K_s,K_t)$ was determined by Zykov~\cite{Zykov}, while the systematic study of generalized Tur\'an numbers for general graphs $H$ and $F$ was initiated by Alon and Shikhelman~\cite{AlonShikhelman}.

For every integer $r\ge 3$, let $C_r$ denote the cycle of length $r$. Generalized Tur\'an problems involving cycles have received considerable attention. An early result of Bollob\'as and Gy\H{o}ri~\cite{BollobasGyori} states that $ \ex(n,C_3,C_5)=\Theta\!\left(n^{3/2}\right).$ Gy\H{o}ri and Li~\cite{GyoriLi} extended this problem to arbitrary forbidden odd cycles and proved that, for every $q\ge 2$, $ \ex(n,C_3,C_{2q+1}) =O_q\bigl(\ex(n,C_{2q})\bigr).$ Together with the Bondy--Simonovits theorem~\cite{BondySimonovits}, this gives
\begin{align}\label{C3}
    \ex(n,C_3,C_{2q+1}) =O_q\!\left(n^{1+1/q}\right).
\end{align}
Several other generalized Tur\'an problems for cycles have also been studied. Gishboliner and Shapira~\cite{GishbolinerShapira} determined, for all distinct fixed integers $r,s$ with $r>3$, the order of magnitude of $ \ex(n,C_r,C_s). $ Related extremal problems for odd cycles were studied by Grzesik and Kielak~\cite{GrzesikKielak} and by Beke and Janzer~\cite{BekeJanzer}. For further results, we refer the reader to the survey of Gerbner and Palmer~\cite{GerbnerPalmer}.

For every integer $r\ge 3$, set
$\mathscr {C}_r:=\{C_3,C_4,\ldots,C_r\},$ and $\mathscr {C}_2:=\varnothing.$
Gerbner et al.~\cite{GGMV} obtained the following conditional result. More precisely, fix an integer $k\ge 2$ and suppose that there exists a family of graphs $\{G_n\}$ satisfying
$|V(G_n)|=n,$ $|E(G_n)|\ge \frac12 n^{1+1/k},$
$G_n\text{ is }\{C_4,C_6,\ldots,C_{2k}\}\text{-free},$
and $d_{G_n}(v)=\frac{2|E(G_n)|}{n}+O_k(1)$ for every $v\in V(G_n).$
Under this hypothesis, they proved that
\[
\ex(n,C_{2k+1},\mathscr {C}_{2k})
=
(1+o(1))\frac{n^{2+1/k}}{4k+2}.
\]
Lazebnik, Ustimenko, and Woldar~\cite{LazebnikUstimenkoWoldar} showed that this hypothesis holds for $k\in\{2,3,5\}$.

More relevant to the present paper, Gerbner et al.~\cite{GGMV} considered the effect of forbidding one additional longer cycle. For integers $l>k\ge 2$, they proved
\[
\Omega_{k,l}\!\left(n^{1+\frac{1}{2l+1}}\right)\le\ex\!\left(n,C_{2k+1},\mathscr {C}_{2k}\cup\{C_{2l}\}\right)\le O_{k,l}\!\left(n^{2-\frac{1}{k+1}}\right),
\]
which gives a subquadratic upper bound when the additional forbidden cycle is even. In contrast, when it is odd, their bounds were
\begin{equation}
\Omega_{k,l}\!\left(
n^{1+\frac{1}{2l+2}}
\right)
\le
\ex\!\left(
n,C_{2k+1},\mathscr {C}_{2k}\cup\{C_{2l+1}\}
\right)
\le
O_{k,l}(n^2).
\label{eq:previous-bounds}
\end{equation}
They conjectured that the quadratic upper bound in \eqref{eq:previous-bounds} can always be improved by a positive power of $n$.

\begin{conjecture}[Gerbner, Gy\H{o}ri, Methuku, and Vizer~\cite{GGMV}]\label{conj:GGMV}
For every pair of integers $l>k\ge 1$, there exists a constant $\varepsilon=\varepsilon(k,l)>0$ such that $ \ex\!\left(n,C_{2k+1},\mathscr {C}_{2k}\cup\{C_{2l+1}\}\right)=O_{k,l}\!\left(n^{2-\varepsilon}\right).$
\end{conjecture}

In this paper, we prove the following result.

\begin{theorem}\label{thm:main}
For all integers $l>k\ge 2$, we have
$$\ex\!\left(n,C_{2k+1},\mathscr {C}_{2k}\cup\{C_{2l+1}\}\right)=O_{k,l}\!\left(n^{\,2-\frac{1}{k(k+1)(l-k)}}\right).$$
\end{theorem}

The case $k=1$ follows from \eqref{C3}. Consequently, Conjecture~\ref{conj:GGMV} holds for all positive integers $k<l$.

\medskip
\noindent\textbf{Notation.}
Throughout the paper, we use the following notation. Let $G=(V(G),E(G))$ be a graph. For $v\in V(G)$, let $N_G(v)$ and $d_G(v)$ denote the neighborhood and degree of $v$, respectively, and let $\delta(G)$ and $\Delta(G)$ denote the minimum and maximum degrees of $G$. For $i\ge 0$, let $N_i^G(v)$ be the set of vertices at distance exactly $i$ from $v$, where the distance is the length of a shortest path and is taken to be $\infty$ if no such path exists. For a graph $H$ and $S\subseteq V(H)$, let $H[S]$ denote the subgraph induced by $S$, and write $e(H):=|E(H)|$. We write $xy$ for the edge $\{x,y\}$, and use $H-xy$ and $H-v$ to denote the deletion of an edge $xy$ and a vertex $v$ together with its incident edges, respectively. The notation $e\ni v$ means that $e$ is an edge incident with $v$. Let $P_r$ denote the path on $r$ vertices.
We will omit the subscript $G$ in all these notations when there is no danger of confusion.

\section{Preliminary results}

We use the following results of Alon, Hoory, and Linial~\cite{AHL} and Erd\H{o}s and Gallai~\cite{ErdosGallai}.

\begin{theorem}[Alon, Hoory, and Linial~\cite{AHL}] \label{thm:AHL}
For every integer $k\ge 2$, $\ex(n,\mathscr {C}_{2k})<\frac12 n^{1+1/k}+\frac12 n.$
\end{theorem}

\begin{theorem}[Erd\H{o}s and Gallai~\cite{ErdosGallai}] \label{thm:EG}
If an $n$-vertex graph contains no $P_r$, then it has at most $(r-2)n/2$ edges.
\end{theorem}

We record the following immediate consequence of the girth condition.

\begin{observation}\label{obs:unique-path}
Let $G$ be an $\mathscr {C}_{2k}$-free graph. Then any two vertices of $G$ are joined by at most one path of length at most $k$.
\end{observation}
Indeed, if two distinct paths of length at most \(k\) joined the same pair of vertices, then their union would contain a cycle of length at most \(2k\), a contradiction.

Fix integers $l>k\ge 2$. Let $G$ be an $(\mathscr {C}_{2k}\cup\{C_{2l+1}\})$-free graph, and let $\mathcal F$ be a family of copies of $C_{2k+1}$ in $G$. For $C\in\mathcal F$ and $v\in V(C)$, let $xy$ be the unique edge of $C$ such that the $v$--$x$ and $v$--$y$ paths in $C-xy$ both have length $k$. We call $xy$ the edge of $C$ opposite $v$. By Observation~\ref{obs:unique-path}, $x,y\in N_k(v)$, and $v$ together with $xy$ uniquely determines $C$.

For $v\in V(G)$, let $L_v=L_v(\mathcal F)$ be the graph on $N_k(v)$ whose edges are the edges opposite $v$ in the cycles of $\mathcal F$ containing $v$. Then $L_v\subseteq G[N_k(v)]$ and
\begin{equation}
e(L_v)=\bigl|\{C\in\mathcal F:v\in V(C)\}\bigr|.
\label{eq:Lv-count}
\end{equation}

\begin{lemma} \label{lem:Lv-sparse}
For every $v\in V(G)$ and every $U\subseteq N_k(v)$, $e(L_v[U])\le 2(l-k)|U|.$ Consequently, $L_v$ admits an orientation with maximum outdegree at most $4(l-k)$.
\end{lemma}

\begin{proof}
Set $d:=l-k\ge 1$, and fix $v\in V(G)$ and $U\subseteq N_k(v)$. For each $x\in N_k(v)$, let $P_x$ be the unique $v$--$x$ path of length $k$. For each $w\in N_G(v)$, define
\[
A_w:=\{x\in N_k(v):P_x\text{ begins with the edge }vw\}.
\]
The nonempty sets $A_w$ form a partition of $N_k(v)$.

We claim that every edge of $L_v$ joins two distinct classes of this partition. Indeed, if $xy\in E(L_v)$, then $xy$ is the edge opposite $v$ on some cycle $C\in\mathcal F$. The $v$--$x$ and $v$--$y$ paths in $C-xy$ are internally vertex-disjoint. By Observation~\ref{obs:unique-path}, they are precisely $P_x$ and $P_y$, and hence they begin with distinct edges incident to $v$. Therefore, every edge of $L_v$ joins two distinct classes of the partition. Consequently, every edge of $L_v[U]$ also joins two distinct classes of the partition.

Now assign each nonempty class independently and uniformly at random to one of two groups $1,2$, and let $U_i$ be the set of vertices of $U$ whose class is assigned to group $i$, for $i\in\{1,2\}$. For any fixed edge $xy\in E(L_v[U])$, its endpoints lie in distinct classes, so $xy$ joins $U_1$ and $U_2$ with probability $1/2$. Thus the expected number of edges between $U_1$ and $U_2$ is $\frac12 e(L_v[U])$. Hence, for some assignment, the spanning subgraph $J$ of $L_v[U]$ consisting of the edges between $U_1$ and $U_2$ satisfies
\[
e(J)\ge \frac12 e(L_v[U]).
\]
Since $(U_1,U_2)$ is a bipartition of $J$, the graph $J$ is bipartite.

We claim that $J$ is $P_{2d+2}$-free. Suppose not, and let $x_0x_1\cdots x_{2d+1}$ be a copy of $P_{2d+2}$ in $J$. Since this path has odd length, its endpoints lie in different parts of $J$; hence $x_0$ and $x_{2d+1}$ belong to classes assigned to different groups, so $P_{x_0}$ and $P_{x_{2d+1}}$ begin with different edges at $v$. By Observation~\ref{obs:unique-path}, they meet only at $v$, for any other common vertex would yield two distinct $v$-paths of length at most $k$. Moreover, every internal vertex of $P_{x_0}$ or $P_{x_{2d+1}}$ has distance at most $k-1$ from $v$, while every $x_i$ lies in $N_k(v)$. Thus the path $x_0x_1\cdots x_{2d+1}$ meets $P_{x_0}\cup P_{x_{2d+1}}$ only at its endpoints. Their union therefore contains a cycle of length $k+(2d+1)+k=2k+2d+1=2l+1$, contradicting the $C_{2l+1}$-freeness of $G$. Thus $J$ is $P_{2d+2}$-free.

By Theorem~\ref{thm:EG}, $e(J)\le \frac{2d}{2}|U|=d|U|$. Consequently, $e(L_v[U])\le 2e(J)\le 2d|U|=2(l-k)|U|$. Since this holds for every $U\subseteq N_k(v)$, every nonempty induced subgraph of $L_v$ has average degree at most $4d$, and thus contains a vertex of degree at most $4d$. We repeatedly delete such vertices from $L_v$ until no vertices remain, and list the vertices in the order of deletion as $z_1,\ldots,z_m$. By construction, for each $i$, the vertex $z_i$ has at most $4d$ neighbors among $z_{i+1},\ldots,z_m$. Orient each edge $z_iz_j$ from $z_i$ to $z_j$ whenever $i<j$. Then the outdegree of $z_i$ is at most $4d$, so the resulting orientation has maximum outdegree at most $4d=4(l-k)$.
\end{proof}

\begin{lemma}
\label{lem:flag-count}
Let $l>k\ge 2$, and let $G$ be an $(\mathscr {C}_{2k}\cup\{C_{2l+1}\})$-free graph on $N$ vertices, with minimum degree $\delta$ and maximum degree $\Delta$. Let $\mathcal F$ be a family of copies of $C_{2k+1}$ in $G$. If $\delta\ge 2(k+l+1),$ then $|\mathcal F|=O_{k,l}\!\left(\frac{N^2\Delta^{l-k-1}}{\delta^{l-k}}\right).$
\end{lemma}

\begin{proof} For each $v\in V(G)$, fix an orientation of $L_v=L_v(\mathcal F)$ with maximum outdegree at most $4(l-k)$, as given by Lemma~\ref{lem:Lv-sparse}.

\begin{claim}\label{claim}
For $a,b\in V(G)$ and $s\ge 1$, let $p_s(a,b)$ denote the number of $a$--$b$ paths of length $s$. Then $p_s(a,b)\le \Delta^{\max\{s-k,0\}}.$ Moreover, for distinct vertices $z,x,w$, the number of $z$--$x$ paths of length $l$ containing $w$ is at most $(l-1)\Delta^{l-k-1}.$
\end{claim}
\begin{proof}
For $s\le k$, the first assertion follows from Observation~\ref{obs:unique-path}. Suppose that $s>k$. There are at most $\Delta^{s-k}$ possibilities for the first $s-k$ edges of an $a$--$b$ path of length $s$. Once this initial segment is fixed, the remaining segment is a path of length $k$ to $b$, and is unique if it exists by Observation~\ref{obs:unique-path}. Hence $p_s(a,b)\le \Delta^{s-k}.$

For the second assertion, suppose that $w$ is the $i$th vertex after $z$ on such a path, where $1\le i\le l-1$. For a fixed $i\in\{1,\ldots,l-1\}$, the first assertion implies that the number of such paths in which $w$ is the $i$th vertex after $z$ is at most
\[
p_i(z,w)p_{l-i}(w,x)
\le
\Delta^{\max\{i-k,0\}+\max\{l-i-k,0\}}.
\]
We claim that
\[
\max\{i-k,0\}+\max\{l-i-k,0\}
\le l-k-1.
\]
If both terms on the left are positive, then their sum is
\[
(i-k)+(l-i-k)=l-2k\le l-k-1.
\]
Otherwise, at least one term is zero. Since $1\le i\le l-1$, the other term is at most $l-k-1$. Summing over the $l-1$ possible positions of $w$, we obtain at most $(l-1)\Delta^{l-k-1}$ such paths, proving the claim.
\end{proof}

A \emph{flag} is a triple $(C,v,R)$, where $C\in\mathcal F$, $v\in V(C)$, and
\[
R=v_0v_1\cdots v_{l -k},
\qquad
v_0=v,
\qquad
V(R)\cap V(C)=\{v\}.
\]
Let $\Phi$ denote the number of flags. For fixed $C$ and $v$, construct $R$ successively. If $v_0,\ldots,v_j$ have been chosen, where $0\le j<l -k$, then
\[
\bigl|V(C)\cup\{v_0,\ldots,v_j\}\bigr|
=
2k+j+1
\le
2k+l-k
=
k+l.
\]
Thus $v_j$ has at least $\delta-(k+l)\ge \frac{\delta}{2} $ available neighbors. Consequently,
\begin{equation}
\Phi
\ge
(2k+1)|\mathcal F|
\left(\frac{\delta}{2}\right)^{l-k}.
\label{eq:flag-lower}
\end{equation}

For each flag $(C,v,R)$, write the edge of $C$ opposite $v$ as $x\to y$ according to the fixed orientation of $L_v$, and assign the flag to the ordered pair$(v_{l-k},x).$ Since $v_{l-k}\notin V(C)$, the two vertices in this ordered pair are distinct.

\begin{claim}
\label{claim:flag-bound}
For any distinct vertices $u,x\in V(G)$, at most $4(l-k)l(l-1)\Delta^{l-k-1}$ flags are assigned to $(u,x)$.
\end{claim}

\begin{proof}
Fix distinct vertices $u,x\in V(G)$, and suppose that at least one flag is assigned to $(u,x)$. Fix one such flag $(C,v,R)$ as a \emph{reference flag}. Write
\[
R=v_0v_1\cdots v_{l-k},
\qquad
v_0=v,
\qquad
v_{l-k}=u,
\]
and let $x\to y$ be the oriented edge of $C$ opposite $v$.

Let $Q$ be the path obtained by traversing $R$ from $u$ to $v$ and then following the $v$--$y$ path of length $k$ in $C-xy$. The reference flag is fixed, and hence so is $Q$. Since $V(R)\cap V(C)=\{v\},$ the path $Q$ is simple and has length $l$. Moreover, $x\notin V(Q)$: the $v$--$y$ path in $C-xy$ does not contain $x$, while $R-v$ is disjoint from $C$.

Now let $(C',v',R')$ be any flag assigned to $(u,x)$. Write
\[
R'=v'_0v'_1\cdots v'_{l-k},
\qquad
v'_0=v',
\qquad
v'_{l-k}=u,
\]
and let $x\to y'$ be the oriented edge of $C'$ opposite $v'$. By traversing $R'$ from $u$ to $v'$ and then following the $v'$--$x$ path of length $k$ in $C'-xy'$, we obtain a simple path $P$ of length $l$ from $u$ to $x$.

We claim that $V(P)\cap\bigl(V(Q)\setminus\{u\}\bigr)\ne\varnothing.$ Indeed, otherwise $P$ and $Q$ would meet only at $u$. Since $x\notin V(Q)$, the paths $P$ and $Q$, together with the edge $xy$, would form a cycle of length $l+l+1=2l+1,$ contradicting the assumption that $G$ is $C_{2l+1}$-free. Since the fixed path $Q$ has length $l$, $|V(Q)\setminus\{u\}|=l.$ Also, every vertex of $V(Q)\setminus\{u\}$ is distinct from both $u$ and $x$. Thus, by Claim~\ref{claim}, for each $w\in V(Q)\setminus\{u\}$, there are at most $(l-1)\Delta^{l-k-1}$ paths of length $l$ from $u$ to $x$ containing $w$. Summing over the $l$ choices of $w$, possibly counting the same path more than once, gives at most $l(l-1)\Delta^{l-k-1}$ possibilities for $P$.

Fix one such path and write
\[
P=p_0p_1\cdots p_l,
\qquad
p_0=u,
\qquad
p_l=x.
\]
For any flag $(C',v',R')$ giving rise to $P$, the construction of $P$ forces
\[
v'=p_{l-k},
\qquad
R'=p_{l-k}p_{l-k-1}\cdots p_0,
\]
and forces the $v'$--$x$ path in $C'-xy'$ to be $p_{l-k}p_{l-k+1}\cdots p_l.$ Hence $P$ uniquely determines $v'$, $R'$, and the $v'$--$x$ path in $C'-xy'$.

It remains to choose $y'$ such that $x\to y'$ is an edge in the fixed orientation of $L_{v'}$. By Lemma~\ref{lem:Lv-sparse}, there are at most $4(l-k)$ choices for $y'$. For each such choice, Observation~\ref{obs:unique-path} implies that $v'$ and the edge $xy'$ determine at most one cycle $C'\in\mathcal F$. Since $R'$ is already fixed, each choice of $y'$ gives at most one flag. Therefore, at most $4(l-k)l(l-1)\Delta^{l-k-1}$ flags are assigned to $(u,x)$.
\end{proof}

Since every flag is assigned to exactly one ordered pair of distinct vertices of $G$, and there are at most $N(N-1)\le N^2$ such ordered pairs, Claim~\ref{claim:flag-bound} gives
\begin{equation}
\Phi\le4(l-k)l(l-1)N^2\Delta^{l-k-1}.
\label{eq:flag-upper}
\end{equation}
Combining \eqref{eq:flag-lower} and \eqref{eq:flag-upper}, we obtain
\[
(2k+1)|\mathcal F|
\left(\frac{\delta}{2}\right)^{l-k}
\le
4(l-k)l(l-1)N^2\Delta^{l-k-1}.
\]
Therefore,
\[
|\mathcal F|
\le
\frac{2^{l-k+2}(l-k)l(l-1)}{2k+1}
\frac{N^2\Delta^{l-k-1}}{\delta^{l-k}}
=
O_{k,l}\!\left(
\frac{N^2\Delta^{l-k-1}}{\delta^{l-k}}
\right),
\]
as required.
\end{proof}

\section{Proof of Theorem \ref{thm:main}}
\label{sec:main-proof}

\begin{lemma}
\label{lem:pruning}
Let $G$ be a graph with $n$ vertices and $m$ edges, and suppose that $G$ contains $M>0$ copies of $C_{2k+1}$. Then there exists a family $\mathcal F$ of at least $M/2$ such cycles such that every vertex contained in a member of $\mathcal F$ is contained in at least $M/(4n)$ members of $\mathcal F$, and every edge contained in a member of $\mathcal F$ is contained in at least $M/(4m)$ members of $\mathcal F$.
\end{lemma}

\begin{proof}
Let $\mathcal A$ initially be the family of all $M$ copies of $C_{2k+1}$ in $G$. Repeatedly perform either of the following operations whenever possible:
\begin{itemize}
\item if a vertex contained in a member of $\mathcal A$ is contained in fewer than $M/(4n)$ members of $\mathcal A$, delete all members of $\mathcal A$ containing that vertex;
\item if an edge contained in a member of $\mathcal A$ is contained in fewer than $M/(4m)$ members of $\mathcal A$, delete all members of $\mathcal A$ containing that edge.
\end{itemize}
The process terminates since each operation deletes at least one cycle. Let $\mathcal F$ be the family remaining at termination.

Each vertex can trigger at most one deletion operation, since no remaining cycle contains it afterward. As each such operation deletes fewer than $M/(4n)$ cycles, all vertex-deletion operations together delete fewer than $n\cdot\frac{M}{4n}=\frac{M}{4}$ cycles. Similarly, each edge can trigger at most one operation, so all edge-deletion operations together delete fewer than $m\cdot\frac{M}{4m}=\frac{M}{4}$ cycles. Thus fewer than $M/2$ cycles are deleted, and hence $|\mathcal F|\ge \frac{M}{2}.$ By the termination condition, every vertex and every edge contained in a member of $\mathcal F$ has the required multiplicity.
\end{proof}

\begin{proof}[\textbf{Proof of Theorem~\ref{thm:main}}]
Let $G$ be an $n$-vertex $(\mathscr {C}_{2k}\cup\{C_{2l+1}\})$-free graph, and set $M:=N(C_{2k+1},G)$ and $m:=e(G).$ There is nothing to prove if $M=0$, so assume that $M>0$. Since $G$ is $\mathscr {C}_{2k}$-free, Theorem~\ref{thm:AHL} gives
\begin{equation}
m=O_k\!\left(n^{1+1/k}\right).
\label{eq:edge-bound}
\end{equation}

Apply Lemma~\ref{lem:pruning}, and let $\mathcal F$ be the resulting family of copies of $C_{2k+1}$. Let $H$ be the union of the cycles in $\mathcal F$, and write
\[
N:=|V(H)|,
\qquad
\delta:=\delta(H),
\qquad
\Delta:=\Delta(H).
\]
Then $N\le n$, and $H$ is also $(\mathscr {C}_{2k}\cup\{C_{2l+1}\})$-free. All degrees and neighborhoods below are taken in $H$.

For $v\in V(H)$ and $e\in E(H)$, let
\[
c(v):=
\bigl|\{C\in\mathcal F:v\in V(C)\}\bigr|,
\qquad
\mu(e):=
\bigl|\{C\in\mathcal F:e\in E(C)\}\bigr|.
\]
Every vertex and every edge of $H$ lies in a member of $\mathcal F$. Thus Lemma~\ref{lem:pruning} gives
\begin{equation}
|\mathcal F|\ge \frac{M}{2},
\qquad
c(v)\ge \frac{M}{4n},
\qquad
\mu(e)\ge \frac{M}{4m}.
\label{eq:multiplicity-lower}
\end{equation}
For every $v\in V(H)$, the vertex $v$ is incident with exactly two edges of each cycle in $\mathcal F$ containing $v$. Hence, by the definitions of $c(v)$ and $\mu(e)$, $\sum_{e\ni v}\mu(e)=2c(v).$

For each $v\in V(H)$, form the auxiliary graph $L_v=L_v(\mathcal F)$ in $H$. Since \(H\) is \(\mathscr {C}_{2k}\)-free and every cycle in \(\mathcal F\) is contained in \(H\), the endpoints of the edge opposite \(v\) lie in \(N_k^H(v)\). Thus the auxiliary graph \(L_v(\mathcal F)\) is well-defined in \(H\). By \eqref{eq:Lv-count} and Lemma~\ref{lem:Lv-sparse}, applied with $U=N_k(v)$,
\begin{equation}
c(v)=e(L_v)
\le 2(l-k)|N_k(v)|
\le 2(l-k)N.
\label{eq:cv-upper}
\end{equation}
It follows from \eqref{eq:multiplicity-lower} and \eqref{eq:cv-upper} that
\[
\frac{M}{4m}d_H(v)
\le
\sum_{e\ni v}\mu(e)
=
2c(v)
\le
4(l-k)N
\le
4(l-k)n.
\]
Taking the maximum over all $v\in V(H)$ yields
\begin{equation}
\Delta
\le
16(l-k)\frac{nm}{M}
=
O_{k,l}\!\left(\frac{nm}{M}\right).
\label{eq:Delta-upper}
\end{equation}

Every vertex of $N_k(v)$ is the endpoint of a path of length $k$ starting at $v$. The first edge of such a path has at most $d_H(v)$ choices, and each of the remaining $k-1$ edges has at most $\Delta$ choices. Hence the number of such paths is at most $d_H(v)\Delta^{k-1}$, and therefore $|N_k(v)|\le d_H(v)\Delta^{k-1}.$ Combining this with \eqref{eq:multiplicity-lower}, \eqref{eq:Lv-count}, and Lemma~\ref{lem:Lv-sparse}, we obtain
\[
\frac{M}{4n}
\le
c(v)
=
e(L_v)
\le
2(l-k)|N_k(v)|
\le
2(l-k)d_H(v)\Delta^{k-1}.
\]
Therefore, $\delta\ge\frac{M}{8(l-k)n\Delta^{k-1}}.$
Using \eqref{eq:Delta-upper}, we get
\begin{equation}
\delta\ge \frac{1}{8(l-k)\bigl(16(l-k)\bigr)^{k-1}}\frac{M^k}{n^km^{k-1}}=\Omega_{k,l}\!\left(\frac{M^k}{n^km^{k-1}}\right).
\label{eq:delta-lower}
\end{equation}

Suppose first that $\delta<2(k+l+1).$ Then \eqref{eq:delta-lower} implies $M^k=O_{k,l}\!\left(n^km^{k-1}\right).$ Taking the $k$th root and applying Theorem~\ref{thm:AHL} through \eqref{eq:edge-bound}, we obtain
\[
M=O_{k,l}\!\left(nm^{(k-1)/k}\right)=O_{k,l}\!\left(n^{2-1/k^2}\right).
\]
Since $l-k\ge1$, $k(k+1)(l-k)\ge k^2,$ and therefore $2-\frac{1}{k^2}\le2-\frac{1}{k(k+1)(l-k)}.$ Thus the required bound holds in this case.

We may therefore assume that $\delta\ge 2(k+l+1).$ Lemma~\ref{lem:flag-count}, applied to $H$ and $\mathcal F$, gives $|\mathcal F|=O_{k,l}\!\left(\frac{N^2\Delta^{l-k-1}}{\delta^{l-k}}\right).$ By ~\eqref{eq:multiplicity-lower}, we have $|\mathcal F|\ge M/2$, and $N\le n$. Hence
\[
M=O_{k,l}\!\left(\frac{n^2\Delta^{l-k-1}}{\delta^{l-k}}\right).
\]
Substituting \eqref{eq:Delta-upper} and \eqref{eq:delta-lower}, we obtain
\begin{align*}
M=O_{k,l}\!\left[n^2\left(\frac{nm}{M}\right)^{l-k-1}\left(\frac{n^km^{k-1}}{M^k}\right)^{l-k}\right] =
O_{k,l}\!\left(n^{(k+1)(l-k)+1}m^{k(l-k)-1}M^{1-(k+1)(l-k)}\right).
\end{align*}
Consequently,
\[
M^{(k+1)(l-k)}=O_{k,l}\!\left(n^{(k+1)(l-k)+1}m^{k(l-k)-1}\right).
\]
Applying Theorem~\ref{thm:AHL} again through \eqref{eq:edge-bound}, we get
\begin{align*}
M^{(k+1)(l-k)}=O_{k,l}\!\left(n^{(k+1)(l-k)+1+(1+1/k)(k(l-k)-1)}\right) =O_{k,l}\!\left(n^{2(k+1)(l-k)-1/k}\right).
\end{align*}
Taking the $(k+1)(l-k)$th root gives $M=O_{k,l}\!\left(n^{\,2-\frac{1}{k(k+1)(l-k)}}\right),$ as required.
\end{proof}

\section{Concluding remarks}

In this paper, we proved that, for all integers $l>k\ge 2$,
\[
\ex\!\left(
n,C_{2k+1},\mathscr {C}_{2k}\cup\{C_{2l+1}\}
\right)
=
O_{k,l}\!\left(
n^{\,2-\frac{1}{k(k+1)(l-k)}}
\right),
\]
thereby confirming Conjecture~\ref{conj:GGMV}.

For fixed integers $l>k\ge 2$, it remains an interesting open problem to determine the orders of magnitude of $\ex\!\left(n,C_{2k+1},\mathscr {C}_{2k}\cup\{C_{2l}\}\right)$ and $\ex\!\left(n,C_{2k+1},\mathscr {C}_{2k}\cup\{C_{2l+1}\}\right).$

\end{document}